\documentclass[11pt, reqno]{amsart}
\usepackage{fullpage}
\usepackage{textcmds}  
\usepackage{amsmath, amssymb, amsfonts, amstext, verbatim, amsthm, mathrsfs}
\usepackage[mathcal]{eucal}
\usepackage{microtype}
\usepackage{graphicx}
\usepackage[all]{xy}
\usepackage[modulo]{lineno}
\usepackage[usenames]{color}
\usepackage{enumitem}
\usepackage{xspace}

\usepackage{aliascnt} 
\usepackage[colorlinks=true,linkcolor=blue,citecolor=blue,urlcolor=blue,citebordercolor={0 0 1},urlbordercolor={0 0 1},linkbordercolor={0 0 1}]{hyperref} 

\usepackage[shortalphabetic]{amsrefs} 

\theoremstyle{plain}

\newcommand{\refnewtheoremn}[4]{%
\newaliascnt{#1}{#2}
\newtheorem{#1}[#1]{#3}
\aliascntresetthe{#1}
\expandafter\providecommand\csname #1autorefname\endcsname{#4}}

\newcommand{\refnewtheorem}[3]{\refnewtheoremn{#1}{#2}{#3}{#3}}

\newtheorem{thm}{Theorem}[section]

\refnewtheorem{lem}{thm}{Lemma}
\refnewtheorem{claim}{thm}{Claim}
\refnewtheorem{cor}{thm}{Corollary}
\refnewtheorem{conj}{thm}{Conjecture}
\refnewtheorem{prop}{thm}{Proposition}
\refnewtheorem{quest}{thm}{Question}
\refnewtheorem{amplif}{thm}{Amplification}
\refnewtheorem{scholium}{thm}{Scholium}

\refnewtheorem{tablethm}{thm}{Table}
\refnewtheorem{assumption}{thm}{Assumption}
\refnewtheorem{conclusion}{thm}{Conclusion}
\refnewtheorem{problem}{thm}{Problem}

\theoremstyle{definition}
\refnewtheorem{rem}{thm}{Remark}
\refnewtheorem{defn}{thm}{Definition}
\refnewtheorem{notn}{thm}{Notation}
\refnewtheorem{ex}{thm}{Example}
\refnewtheorem{fact}{thm}{Fact}

{\end{minipage} \end{center}}

\newcommand{\bA}{\mathbf{A}}
\newcommand{\cA}{\mathcal{A}}

\newcommand{\cB}{\mathcal{B}}

\newcommand{\cC}{\mathcal{C}}

\newcommand{\bD}{\mathbf{D}}

\newcommand{\bG}{\mathbf{G}}
\newcommand{\cG}{\mathcal{G}}

\newcommand{\bL}{\mathbf{L}}

\newcommand{\cO}{\mathcal{O}}

\newcommand{\cS}{\mathcal{S}}

\newcommand{\cU}{\mathcal{U}}

\newcommand{\cV}{\mathcal{V}}

\newcommand{\cX}{\mathcal{X}}

\newcommand{\cY}{\mathcal{Y}}

\newcommand{\bZ}{\mathbf{Z}}
\newcommand{\cZ}{\mathcal{Z}}

\newcommand{\fg}{\mathfrak{g}}

\newcommand{\fu}{\mathfrak{u}}

\newcommand{\id}{\mathop{{\rm id}}\nolimits}

\renewcommand{\bZ}{{\mathbb Z}}

\newcommand{\op}[1]{\!\!\mathop{\rm ~#1}\nolimits}

\newcommand{\Gm}{\mathbb{G}_m}
\newcommand{\bdot}{\centerdot}
\newcommand{\radj}[1]{\beta^{\geq #1}}
\newcommand{\ladj}[1]{\beta^{< #1}}
\newcommand{\dual}{\vee}
\newcommand{\sod}[1]{\langle #1 \rangle}

\newcommand{\QC}{\op{QC}}
\newcommand{\dbc}{D^b \op{Coh}}

\newcommand{\icoh}{\op{QC}^!}
\newcommand{\Perf}{\op{Perf}}

\newcommand{\X}{\cX}
\newcommand{\Y}{\cY}
\renewcommand{\S}{\cS}
\newcommand{\Z}{\cZ}
\newcommand{\V}{\cV}
\newcommand{\U}{\cU}
\newcommand{\Spec}{\op{Spec}}
\newcommand{\inner}[1]{\underline{#1}}
\newcommand{\Coh}{\op{Coh}}
\newcommand{\Map}{\op{Map}}
\newcommand{\heart}{\heartsuit}
\newcommand{\RHom}{\op{RHom}}
\newcommand{\colim}{\op{colim}}

\newcommand{\rgs}[1]{R\inner{\Gamma}_\S^{#1}}
\newcommand{\Mod}{\text{\rm{-Mod}}}
\newcommand{\rank}{\op{rank}}
\newcommand{\cofib}{\op{cofib}}

\newcommand{\Qshriek}{{Q!}}

\begin{document}

\title{Remarks on $\Theta$-stratifications and derived categories}
\author{Daniel Halpern-Leistner}

\begin{abstract}
This note extends some recent results on the derived category of a geometric invariant theory quotient to the setting of derived algebraic geometry. Our main result is a structure theorem for the derived category of a derived local quotient stack which admits a stratification of the kind arising in geometric invariant theory. The use of derived algebraic geometry leads to results with pleasingly few hypotheses, even when the stack is not smooth. Using the same methods, we establish a ``virtual non-abelian localization theorem'' which is a K-theoretic analog of the virtual localization theorem in cohomology.
\end{abstract}

\maketitle

When $X$ is a smooth projective-over-affine variety over a field of characteristic 0, and $G$ is a reductive group acting on $X$, goemetric invariant theory provides a $G$-equivariant stratification $X = X^{ss} \cup S_0 \cup \cdots \cup S_N$, where $X^{ss}$ is the \emph{semistable locus} and $X^{us} = S_0 \cup \cdots \cup S_N$ is the \emph{unstable locus}. The main result of \cite{halpern2012derived} provides a semiorthogonal decomposition of the derived category of equivariant coherent sheaves,
\begin{equation}\label{eqn:basic_decomp}
D^b\Coh(X/G) = \sod{D^b\Coh_{X^{us}}(X/G)^{<w},D^b\Coh(X^{ss}/G),D^b\Coh_{X^{us}}(X/G)^{\geq w}},
\end{equation}
where the categories $D^b\Coh_{X^{us}}(X/G)^{<w, \geq w}$ consist of objects supported on $X^{us}$, and in fact these categories admit infinite semiorthogonal decompositions further refining the structure of $D^b\Coh(X/G)$. This decomposition is also implicit in the main result of \cite{ballard2012variation}. One can think of \autoref{eqn:basic_decomp} as a week kind of direct sum decomposition, which categorifies classically studied direct sum decompositions of equivaraint cohomology and topological K-theory (with respect to a maximal compact subgroup $G^c \subset G$) \cite{harada2007surjectivity},
\[K_{G^c}(X) \simeq K_{G^c}(X^{ss}) \oplus K_{G^c}(S_0) \oplus \cdots \oplus K_{G^c}(S_N)\]

The final version of \cite{halpern2012derived} proves a version of \autoref{eqn:basic_decomp} for singular classical global quotient stacks $\X = X/G$, but only under two additional technical hypotheses, referred to as (L+) and (A). Unfortunately these hypotheses often fail, even for $\X$ with local complete intersection singularities. The main observation of this note is that by passing to the setting of derived algebraic geometry, there is a version which holds generally, without any technical hypotheses. In particular, it applies even to classical quotient stacks in situations where the main theorem of \cite{halpern2012derived} does not apply.

Generalizing the GIT stratification of a smooth quotient stack, we introduce the notion of a derived $\Theta$-stratification in a derived quotient stack, $\X = X/G$, in characteristic $0$ (\autoref{defn:theta_stratum}). Rather than working with $D^b\Coh(\X)$, our main structure theorem, \autoref{thm:derived_Kirwan_surjectivity}, provides a semiorthogonal decomposition of $D^-\Coh(\X)$ generalizing \autoref{eqn:basic_decomp} for a single derived $\Theta$-stratum $\S \subset \X$. We provide two more refined versions of this result: first for $D^b\Coh(\X)$ when $\X$ is quasi-smooth (and certain obstructions vanish) in \autoref{thm:derived_Kirwan_surjectivity_quasi-smooth}, which has applications to variation of GIT quotient (see \autoref{cor:wall_crossing}), and second for $\Perf(\X)$ when the inclusion $\S \subset \X$ is a regular embedding (\autoref{prop:DKS_perfect}). Finally, we show how to extend the main structure theorem to the setting of multiple strata on a local quotient stack in \autoref{thm:derived_Kirwan_surjectivity_full}.

As another application of the notion of a derived $\Theta$-stratum, we establish a $K$-theoretic virtual non-abelian localization theorem, \autoref{thm:nonabelian_localization}. Consider a quotient stack $X /T$, where $T$ is a torus, and $X$ is given a $T$-equivariant perfect obstruction theory, and let $X_i$ denote the components of the fixed locus $X^T$. Then virtual localization in cohomology, \cite{graber1999localization}, provides a method for computing integrals of equivariant cohomology classes
$$\int_{[X]^{vir}} \eta = \sum_i \int_{[X_i]^{vir}} \frac{\eta|_{X_i/T}}{e(N_i^{vir})}$$
in a localization of the power series ring $H^\ast(BT)$.

The K-theoretic localization theorem is an analgous expression for the $K$-theoretic integral $\chi(\X,F) := \sum (-1)^n R^n\Gamma(X,F)^G$ for $F \in \Perf(X/G)$. Recall that each stratum $S_i$ comes with a distinguished one-parameter-subgroup $\lambda_i$ and fixed component $Z_i \subset S_i^{\lambda_i}$. We denote the map of stacks $\sigma_i : \Z_i := Z_i/L_i \to X/G$, where $L_i$ is the centralizer of $\lambda$. The non-abelian localization formula has the form
\begin{equation} \label{eqn:basic_localization}
\chi(\X,F) = \chi(\X^{ss},F) + \sum \chi(\Z_i,E_i \otimes \sigma^\ast F),
\end{equation}
where $E_i$ are certain quasi-coherent sheaves playing the role of the reciprocal of the Euler class in $K$-theory (see \autoref{thm:nonabelian_localization}).

The $E_i$ are infinite direct sums of complexes of coherent sheaves, but only finitely many of these contribute to the Euler characteristic $\chi$, so the expression is well-defined. Furthermore, the objects $E_i$ depend on the one-parameter-subgroups $\lambda_i$ in addition to the normal bundles of $Z_i$. In a way this is a strength of the formula: In the case of a $\Gm$-action the fixed loci $\Z_i$ do not depend on the stratification, but the classes $E_i$ do. Thus one can use \autoref{eqn:basic_localization} to compare $\chi(\X,F)$ and $\chi(\X^{ss},F)$, or one could choose a stratification for which $\X^{ss} = \emptyset$, in which case \autoref{eqn:basic_localization} provides an expression for $\chi(\X,F)$ in terms of ``easier'' integrals over the fixed loci $\Z_i$.

To the author's knowlege, \autoref{eqn:basic_localization} first appears in \cite{teleman2009index} for the case of smooth local quotient stacks. The formula requires no modification in the case of a quasi-smooth stack, but the objects $E_i$ must be appropriately interpreted (the ``virtual fundamental class'' is implicit in the fact that $R\Gamma(\X,\bullet)$ depends on the derived structure of $\X$). Constantin Teleman and Chris Woodward suspected a version for quasi-smooth $\X$, and a version of the formula appears in \cite{gonzalez2013gauged} for the moduli of stable curves in a smooth projectively embedded $G$-variety, although the proof there is not entirely correct.\footnote{Theorem 5.5 of \cite{gonzalez2013gauged} is essentially our localization formula stated for the specific case of the moduli space of stable curves. The key ingredient of the proof there, Proposition 5.2, only treats the case where the inclusions of the strata $\S_i \hookrightarrow \X$ are regular embeddings, and when the natural projections $\S_i \to \Z_i$ are vector bundles. However, the treatment there is vague as to the derived structure on the strata themselves, and one has to be careful in the case of quasi-smooth stacks. Either the vector bundle condition or the regular embedding condition can fail when $\X$ is quasi-smooth, and in fact both conditions hold essentially only if the derived obstruction space at each point of $\Z_i$ has weight $0$ with respect to $\lambda_i$ (See \autoref{ex:derived_stratum} below). The complete proof requires a bit more care as to the derived structure of the strata, which is what the notion of a derived $\Theta$-stratum accomplishes.} The notion of a derived $\Theta$-stratum is necessary in order to establish the virtual non-abelian localization theorem in the generality needed for its full range of applications.

\begin{rem}
This material will eventually be subsumed by a larger project studying the structure on derived categories of quasi-geometric stacks induced by $\Theta$-stratifications \cite{halpern2015derived}. That paper will prove theorems analogous to \autoref{thm:derived_Kirwan_surjectivity}, \autoref{thm:derived_Kirwan_surjectivity_quasi-smooth}, \autoref{thm:derived_Kirwan_surjectivity_full}, and \autoref{thm:nonabelian_localization} for $\Theta$-stratifications in stacks which are not local quotient stacks, and the proofs will make a more intrinsic (and essential) use of the modular interpretation provided by \cite{halpern2014structure}.

The proofs in this paper closely follow those in the final version of \cite{halpern2012derived}. While these results are not as general or complete as the one which will eventually appear in \cite{halpern2015derived}, the methods here have the advantage of being more concrete -- they involve mostly explicit manipulations of complexes of vector bundles. Therefore we feel this note will serve as a useful counterpart to \cite{halpern2015derived}. In addition, in the smooth local quotient setting, Matthew Ballard has already found interesting applications of these ideas to the moduli of semistable sheaves on surfaces \cite{ballard2014wall}, and we hope that the derived version might prove useful for the study of local-quotient moduli stacks as well.
\end{rem}

\subsubsection{Author's note}

I would like to thank Constantin Teleman for first suggesting the virtual non-abelian localization theorem as my thesis problem. I would also like to thank Davesh Maulik and the rest of the faculty of the algebraic geometry group at Columbia for encouraging me to continue this line of research. I would like to thank Chris Woodward for his encouragement and for his comments on an early version of this note. This research was supported by Columbia University and the Institute for Advanced Study, as well as an NSF Postdoctoral Research Fellowship.

\tableofcontents

\section{Definitions}

First we establish notation, and introduce certain subcategories of the derived category of quasicoherent sheaves which will be used to establish our main structure theorem,  \autoref{thm:derived_Kirwan_surjectivity}.

\begin{notn}
All of our stacks will be derived algebraic stacks, i.e. presheaves on the $\infty$-category $dg-Alg_k^{op}$, of commutative connective differential graded algebras over a fixed field, $k$, of characteristic $0$. We require that they are sheaves in the \'{e}tale topology, are $1$-stacks, and admit a smooth representable morphism from an affine scheme. (See \cite[Section 3]{lurie2011DAG8} or \cite[Chapter 2.2]{toen2008homotopical} -- the theories agree over a field of characteristic 0)
\end{notn}

More concretely, we will only consider quotient stacks and local quotient stacks. Meaning that Zariski locally, $\X$ admits an affine morphism to a quotient stack $\X' := X/G$, where $X$ is smooth and quasiprojective, and $G$ acts linearly. It follows that locally $\X = R\inner{\Spec}_X \cA / G$ for some quasicoherent sheaf of connective CDGA's $\cA$. We also assume that $\X$ is locally finitely presented. An example would be when $\X$ is a derived closed substack of $\X'$.

Let $\lambda$ be a 1PS and let $S = S_\lambda \subset X$ be a KN-stratum for the action of $G$.\footnote{The notion of a KN-stratum was introduced by Teleman \cite{teleman2000quantization} as an abstraction of the stratification in GIT studied by Kirwan, Ness, Kempf, and Hesselink. See \cite{halpern2012derived} for the precise definition we are using.} We denote $\S' = S/G \hookrightarrow \X'$, and we replace $\cA$ with a semifree resolution of the form
$$\cO_X[U_0,U_1,\ldots; d] \xrightarrow{\simeq} \cA$$
where $U_i$ is an equivariant locally free sheaf in homological degree $i$. When working with a global quotient stack, we will regard this presentation as fixed once and for all. We will make an exception when we consider local quotient stacks in \autoref{sect:multiple_strata}, where we will spell out explicit compatibility conditions between local quotient coordinate charts. 

First consider $\V := \X \times_{\X'} \S' \simeq (R\inner{\Spec}_S \cO_S \otimes_{\cO_X} \cA)  / G$. Then letting $P$ be the parabolic subgroup defined by $\lambda$, we can identify $S / G \simeq Y / P$, and thus we can write this locally free sheaf of CDGA's as a locally free sheaf of $P$-equivariant CDGA's, $\tilde{\cA}$, on $Y$. The subsheaf $\tilde{\cA}_{\lambda \geq 1} \subset \tilde{\cA}$ is $P$-equivariant, and we let $\cA_\lambda := \tilde{\cA} / \tilde{\cA} \cdot \tilde{\cA}_{\lambda \geq 1}$. Note that this has an explicit presentation of the form
\[
\cO_Y \left[(U_0|_Y)_{<1},(U_1|_Y)_{<1},(U_2|_Y)_{<1},\ldots ; d|_Y \right] \xrightarrow{\simeq} \cA_\lambda.
\]
Where $U_i|_Y$ admits a canonical short exact sequence of the form $0 \to (U_i|_Y)_{\geq 1} \to U_i|_Y \to (U_i|_Y)_{<1} \to 0$ from the theory for classical stacks developed in \cite{halpern2012derived}.

Next we and let $L$ be the centralizer of $\lambda$ (i.e. the Levi quotient of $P$) and consider $\cO_Z \otimes_{\cO_Y} \cA_\lambda$, regarded as an $L$-equivariant quasicoherent CDGA on $Z$. The sub-module $(\cO_Z \otimes_{\cO_Y} \cA_\lambda)_{<0}$ is $L$-equivariant and a differential ideal (because the differential has degree 0), and we let $\cB_\lambda := \cO_Z \otimes_{\cO_Y} \cA_\lambda / (\cO_Z \otimes_{\cO_Y} \cA_\lambda)_{<0}$ be the quasicoherent CDGA on $Z / L$.

\begin{defn} \label{defn:theta_stratum}
We define the derived $\Theta$-stratum to be $\S := R\inner{\Spec}_{Y} \cA_\lambda / P$. It is a closed substack of $\X$. We also define $\Z = R\inner{\Spec}_Z \cB_\lambda / L$. We have a canonical projection $\pi : \S \to \Z$ with a section $\sigma : \Z \to \S$. Furthermore, this diagram sits above the corresponding diagram for the classical KN-stratum in $\X' = X/G$:
$$\xymatrix{ \Z \ar@/^/[r]^\sigma \ar[d] & \S \ar[r]^i \ar@/^/[l]^\pi \ar[d] & \X \ar[d] \\ \Z' \ar@/^/[r] & \S' \ar[r] \ar@/^/[l] & \X' }$$
where the vertical morphisms are affine.
\end{defn}

\begin{rem}
The derived structure on the $\Theta$-stratum is that of the mapping stack from $\Theta := \bA^1 / \bG_m$ to $\X$. The stacks $\Map(\Theta,\X)$ and $\Map(B\bG_m, \X)$ are algebraic by the main result of \cite{halpern2014mapping}. The underlying classical stack of $\Z$ and $\S$ clearly form the structure of a KN stratum in the underlying classical stack of $\X$, and thus by the modular interpretation of classical KN stratifications in \cite{halpern2014structure}, it follows that $\S^{cl}$ is isomorphic to an open substack of $\Map(\Theta,\X)^{cl} = \Map(\Theta,\X^{cl})^{cl}$. To verify that the derived structures agree, it suffices to check that the cotangent complex of $\S$ as computed in \autoref{lem:relative_cotangent_complex} below agrees with the cotangent complex of the mapping stack.
\end{rem}

\begin{ex} \label{ex:derived_stratum}
Consider the algebra $\cO_X = k[x_1,\ldots,x_n,y_1,\ldots,y_m]/(f(x,y))$ regarded as the coordinate ring of an affine scheme with $\Gm$-action by equipping $x_i$ with positive weights and $y_i$ with negative weights in such a way that $f$ is homogeneous of weight $a$. Then $\cO_S = \cO_X / (x_1,\ldots,x_n)$ is a classical KN-stratum. On the other hand, we can consider the semi-free resolution $\cO_X \sim k[x_i,y_j,u;d]$, where $u$ is a variable of weight $a$ and homological degree $1$ with $du = f(x,y)$. If $a>0$, then the derived $\Theta$-stratum is $\Spec k[y_j] / \Gm$ and is classical -- in this case the stratum is a bundle of affine spaces over $B\Gm$, but $i : \S \hookrightarrow \X$ is not a regular embedding. If $a\leq 0$, then $\S = \Spec k[y_j,u;d] / \Gm$, with $du = f(0,y)$. In this case $i$ is a regular embedding, but if $a<0$, $\Z \simeq B\Gm$ and so $\pi : \S \to \Z$ is not smooth. In the case $a=0$, then $\Z = \Spec k[u;du=0] / \Gm$ and $\pi : \S \to \Z$ is a bundle of affine spaces.
\end{ex}

\begin{notn}
For any stack $\X$, we let $\QC(\X)$ denote the $\infty$-category of (unbounded) quasicoherent complexes on $\X$ (i.e. the $\infty$-categorical version of the derived category of quasicoherent sheaves), defined as the limit under pullback of the categories $R\Mod$ over all maps $\Spec R \to \X$. $D^? \Coh(\X)$, where $? = b,+,-,$ or blank, will denote the full subcategory of $\QC(\X)$ whose homology sheaves are coherent and bounded,\footnote{We have used the classical notation for these categories because it should be familiar to more readers. Note, however, that $D^b\Coh(\X)$ is not the derived category of the abelian category of coherent sheaves on $\X$, because the category of coherent sheaves always agrees with that of the underlying classical stack.} homologically bounded above, homologically bounded below, and unbounded respectively. $\Perf(\X)$ denotes the category of perfect complexes. Adding the subscript $\Coh_\S$ or $\QC_\S$ will refer to the full subcategory of objects set-theoretically supported on the closed substack $\S \subset \X$.
\end{notn}

\begin{notn}
Objects of $\QC(\X)$ will typically be denoted in Roman font, i.e. $F \in \QC(\X)$, and we reserve the notation $F_\bdot$ for when it is necessary to emphasize the underlying complex of $F$. We use homological grading conventions throughout, so for the usual $t$-structure on $\QC(\X)$, $\tau_{\leq n} F$ is an object which is homologically bounded above ($H_i(F) = 0$ for $i>n$) and $\tau_{\geq n}$ is homologically bounded below. Subcategories defined by the $t$-structure will be denoted by subscripts, so $\QC(\X)_{<\infty}$ is the full subcategory of complexes which are homologically bounded above, and $\QC(\X)_\heart$ is the category of quasicoherent sheaves.
\end{notn}

\begin{defn} Regarding $F \in \QC(\Z)$ as an $L$-equivariant $\cO_Z$-module, canonically decomposes into a direct sum of objects based on the weights w.r.t. $\lambda$. We will use this to define the following subcategories of $\QC(\S)$:
\begin{align*}
\Perf(\S)^{\geq w} &= \left\{ F \in \Perf(\S) \left| \sigma^\ast F \in \QC(\Z)^{\geq w} \right.\right\} \\
\Perf(\S)^{<w} &= \left\{ F \in \Perf(\S) \left| \sigma^\ast F \in \QC(\Z)^{< w} \right.\right\} \\
\end{align*}
We let $\QC(\S)^{\geq w}$ (respectively $\QC(\S)^{<w}$) be the smallest stable subcategory of $\QC(\S)$ containing $\Perf(\S)^{\geq w}$ (respectively $\Perf(\S)^{<w}$) and closed under colimits. Furthermore, for any full subcategory $\cC \subset \QC(\S)$, such as $D^-\Coh(\S)$, we let
$$\cC^{\geq w} := \cC \cap \QC(\S)^{\geq w}, \quad \text{and} \quad \cC^{<w} := \cC \cap \QC(\S)^{<w}.$$
\end{defn}

The pushforward functor $i_\ast : \QC(\S) \to \QC(\X)$ admits a right adjoint, which we denote by $i^\Qshriek : \QC(\X) \to \QC(\S)$. Note that this functor is continuous if and only if $\cO_\S$ is perfect as an $\cO_\X$-module, in which case $i_\ast$ preserves perfect complexes. The notation $i^!$ shall be used for the corresponding right adjoint functor between categories of ind-coherent sheaves,\footnote{See \cite{gaitsgory2013ind} for a general definition of $\QC^!$ of a locally almost finitely presented prestack, where it is denoted $\op{IndCoh}(\X)$, and see \cite{drinfeld2013some} for a proof that $\QC^!(\X) = \op{Ind}(\dbc(\X))$ for a quasi-compact algebraic stack with affine stabilizer groups.} which we denote $\QC^!(\X)$. Because $i_\ast : \icoh(\S) \to \icoh(\X)$ preserves $\dbc$, the functor $i^! : \icoh(\X) \to \icoh(\S)$ is continuous.

\begin{defn} \label{defn:general_categories}
We consider the following full subcategories of $\QC(\X)$:
\begin{align*}
\QC(\X)^{\geq w} := \left\{ F \in \QC(\X) | i^\ast F \in \QC(\S)^{\geq w} \right\} \\
\QC(\X)^{< w} := \left\{ F \in \QC(\X) | i^\Qshriek F \in \QC(\S)^{< w} \right\}
\end{align*}
And we follow the convention that for any full subcategory, $\cC \subset \QC(\X)$, we let $\cC^{\geq w}$ (respectively $\cC^{<w}$) denote $\cC \cap \QC(\X)^{\geq w}$ (respectively $\QC(\X)^{<w}$). We define
$$\cG_w := D^-\Coh(\X)^{\geq w} \cap D^-\Coh(\X)^{<w}.$$
\end{defn}

Finally, we will use the following
\begin{notn}
A \emph{baric decomposition} of a stable $\infty$-category $\cC$ will denote a family of semiorthogonal decompositions $\cC = \sod{\cC^{<w},\cC^{\geq w}}$, indexed by $w \in \bZ$, such that $\cC^{<w} \subset \cC^{<w+1}$ and $\cC^{\geq w} \subset \cC^{\geq w-1}$ for all $w$. In this case we define $\cC^w = \cC^{\geq w} \cap \cC^{<w+1}$. We say that the decomposition is \emph{complete} if $\bigcap_w \cC^{<w} = 0$ and $\bigcap_w \cC^{\geq w} = 0$.
\end{notn}

\section{A structure theorem for the derived category}

Our main result of this section will be the following

\begin{thm}\label{thm:derived_Kirwan_surjectivity}
For any $w \in \bZ$, there is a semiorthogonal decomposition
\begin{equation} \label{eqn:main_SOD}
D^- \Coh(\X) = \sod{D^- \Coh_\S(\X)^{<w}, \cG_w,D^-\Coh_\S(\X)^{\geq w}}
\end{equation}
where $\cG_w$ is identified with $D^- \Coh(\X^{ss})$ via the restriction functor. Furthermore, $D^- \Coh_\S(\X)^{<w}$ and $D^-\Coh_\S(\X)^{\geq w}$ give a complete baric decomposition,
\[D^-\Coh_\S(\X) = \sod{D^- \Coh_\S(\X)^{<w}, D^-\Coh_\S(\X)^{\geq w}}, \]
and we can identify $D^-\Coh_\S(\X)^w$ with the essential image of the fully faithful functor $i_\ast \pi^\ast : D^-\Coh(\Z)^w \to D^-\Coh(\X)$.
\end{thm}

The proof follows the proof of the main theorem of \cite{halpern2012derived} closely. The arguments are improved, however, by working with the derived $\Theta$-stratum, which can have a nontrivial derived structure even when $\X$ is classical (as in \autoref{ex:derived_stratum}). In addition, working with $D^- \Coh$ instead of $D^b\Coh$ leads to technical simplifications in the arguments. We will return to the structure $D^b \Coh$ in the next section.

\begin{lem} \label{lem:baric_decomp}
We have baric decompositions $\cC = \sod{\cC^{\geq w},\cC^{<w}}$, where $\cC = \Perf(\S), D^-\Coh(\S)$, or $\QC(\S)$ which are compatible with the inclusions. Furthermore we have alternative descriptions
\begin{align*}
D^-\Coh (\S)^{\geq w} &= \left\{ F \in D^-\Coh(\S) \left| \sigma^\ast F \in D^- \Coh(\Z)^{\geq w} \right. \right\}, \text{ and} \\
D^-\Coh (\S)^{< w} &= \left\{ F \in D^-\Coh(\S) \left| \sigma^\ast F \in D^- \Coh(\Z)^{<w} \right. \right\}.
\end{align*}
\end{lem}
\begin{proof}
The proof is almost identical to \cite[Proposition 3.9]{halpern2012derived}. First one proves the result for $D^- \Coh(\S) \simeq D^- \Coh( \cA_\lambda)$. Instead of resolutions by right-bounded complexes of locally free sheaves, one uses the fact that any $F$ has a right-bounded presentation of the form $\cA \otimes E_\bullet$ as an $\cA$-module (ignoring the differential) where $E_n$ is a locally free sheaf in homological degree $n$ which is pulled back from $\Z'$ under the composition $\S \to \Z \to \Z'$.

Objects of the form $\cA \otimes (E_\bullet)^{\geq w}$ restricted to the underlying classical KN stratum, $\S^{cl}$, lie in $D^-\Coh(\S^{cl})^{\geq w}$, and such objects are left orthogonal to objects in $\QC(\S^{cl})^{<w}_{<\infty}$ by \cite[Lemma 3.14]{halpern2012derived}. Note that any $F \in \QC(\S)_\heart$ is pushed forward from $\S^{cl}$, so $D^-\Coh(\S)^{\geq w}$ is left semi-orthogonal to $\QC(\S)_\heart^{<w}$. For any object $F$ and semi-free presentation as above, $\cA \otimes E_\bullet$, we define the functors $\radj{w} F$ and $\ladj{w}F$ as terms in the exact triangle
$$\xymatrix{ 0 \ar[r] & \radj{w} F \ar@{=}[d] \ar[r] & F \ar@{=}[d] \ar[r] & \ladj{w} F \ar[r] \ar@{=}[d] & 0 \\ 0 \ar[r] & \cA \otimes (E_\bullet)^{\geq w} \ar[r] & \cA \otimes E_\bullet \ar[r] & \cA \otimes (E_\bullet)^{<w} \ar[r] & 0 }.$$
The above observation implies that regardless of the presentations chosen, $\RHom_{\QC(\S)}(\radj{w} F, \ladj{w} G)$ vanishes for any pair of objects.

It is clear from the explicit construction that any object of the form $\radj{w} F$ (respectively $\ladj{w} F$) can be written as a colimit of objects in $\Perf(\S)^{\geq w}$ (respectively $\Perf(\S)^{<w}$). Furthermore objects of the form 
$\radj{w} F$ and $\ladj{w} F$ generate $D^- \Coh(\S)^{\geq w}$ and $D^- \Coh(\S)^{<w}$ respectively, so we have the semiorthogonal decomposition of $D^-\Coh(\S)$. It follows from the construction that $\sigma^\ast \radj{w} F \simeq (\sigma^\ast F)^{\geq w}$ and likewise for $\ladj{w}$. Using this and Nakayama's lemma, one can deduce the alternate characterizations of $D^-\Coh(\S)^{\geq w}$ and $D^-\Coh(\S)^{<w}$.

The argument for why this semiorthogonal decomposition preserves perfect objects follows verbatim from the argument in \cite{halpern2012derived}: essentially that the set of points in $\S$ for which an object in $D^- \Coh(\S)$ is perfect is open, so we know that $\sigma^\ast \radj{w} F \simeq (\sigma^\ast F)^{\geq w}$ is perfect, and the only open substack of $\S$ through which $\sigma :\Z \to \S$ factors is $\S$ itself. Finally because $\S$ is a perfect stack, once we have a semiorthogonal decomposition for $\Perf(\S)$, there is a unique semiorthogonal decomposition of $\QC(\S) = \op{Ind} \Perf(\S)$ such that $\radj{w}$ and $\ladj{w}$ are cocontinuous.
\end{proof}

\begin{rem}\label{rem:homology}
An object in $D^b \Coh(\S)_\heart$ which has weights $<w$ when regarded as an $L$-equivariant $\cO_Z$-module clearly has a resolution by sheaves of the form $\cA \otimes E$ with $E$ locally free of fiber weight $<w$, and thus lies in $D^-\Coh(\S)^{<w}$. Conversely, an object in $D^- \Coh(\S)^{<w}$ has a presentation of this form, and thus has homology sheaves which have negative weights. Thus an object of $D^-\Coh(\S)$ lies in $D^-\Coh(\S)^{<w}$ if and only if its homology lies in that subcategory. Because $\QC(\S) = \op{Ind}(\Perf(\S))$, this description extends to the baric decomposition of $\QC(\S)$ as well: $F \in \QC(\S)^{<w}$ if and only if $H_i(F) \in \QC(\S)^{<w}$ for all $i$.
\end{rem}

The following observation is the primary reason for introducing derived algebraic geometry into this story. It is the direct derived analogy to the computation of the normal bundle of a KN stratum in a smooth variety.

\begin{lem} \label{lem:relative_cotangent_complex}
There is an equivalence of canonical exact triangles in $D^- \Coh(\S)$
\[
\xymatrix{ \radj{1} i^\ast \bL_\X \ar[r] \ar[d] & i^\ast \bL_\X \ar[r] \ar@{=}[d] & \ladj{1} i^\ast \bL_\X \ar[r] \ar[d] & \\ \bL_{\S/\X}[-1] \ar[r] & i^\ast \bL_\X \ar[r] & \bL_\S \ar[r] & }
\]
In particular, $\bL_{\S / \X} \in D^- \Coh (\S)^{\geq 1}$.
\end{lem}

\begin{proof}
It suffices to show that $\sigma^\ast \bL_{\S/\X} \in D^- \Coh(\Z)^{\geq 1}$ and $\sigma^\ast \bL_\S \in D^- \Coh(\Z)^{<1}$. The closed immersion, $i$, was defined as the composition of two closed immersions: first from $\cV := \X \times_{\X'} \S' \to \X = R\inner{\Spec} \tilde{\cA} / P$, and then from $\S \simeq R\inner{\Spec}_S \cA_\lambda / P \to \cV$. Consider the diagram
$$\xymatrix{\S \ar[r] \ar[dr] & \cV \ar[r] \ar[d] & \X \ar[d] \\ & \S' = S/G \ar[r] & \X'=X/G}.$$

To show that $\sigma^\ast \bL_{\S/\X} \in D^- \Coh(\Z)^{\geq 1}$: From the exact triangle for $\bL_{\S/\X}$, it suffices to show that $\bL_{\S/\V}$ and $\bL_{\V/\X}|_\S$ lie in $D^-\Coh(\S)^{\geq 1}$. The former comes from the explicit description of $\cA_\lambda$ as the quotient of $\tilde{\cA}$ by the ideal generated by elements of positive weights, and the latter comes from the fact that the lemma is known for a KN stratum in a smooth classical stack, and $\bL_{\cV/\X} \simeq \bL_{\S' / \X'}|_\V$.

To show that $\sigma^\ast \bL_\S \in D^- \Coh(\Z)^{<1}$: Again from a canonical exact triangle it suffices to show that $\sigma^\ast \bL_{\S / \S'}$ and $\sigma^\ast \bL_{\S'}|_\S$ lie is $D^-\Coh(\Z)^{<1}$. The latter follows from the fact that the lemma is known for a KN stratum in a smooth stack, and the former follows from the explicit description of $\S = R\inner{\Spec}_S \cA_\lambda / P$, where $\cA_\lambda$ is an algebra admitting a presentation by locally free sheaves with non-positive weights.
\end{proof}

\begin{rem}
This computation of the cotangent complex follows more directly from the modular interpretation of $\S$ as a connected component of the mapping stack $\Map(\Theta,\X)$.
\end{rem}

\begin{lem} \label{lem:baric_decomposition_supports}
We have a baric decomposition
$$D^-\Coh_\S(\X) = \sod{D^-\Coh_\S(\X)^{<w}, D^-\Coh_\S(\X)^{\geq w}}.$$
\end{lem}

\begin{proof}
First we note that $i_\ast D^-\Coh(\S)^{\geq w}$ is left semi-orthogonal to $i_\ast D^-\Coh(\S)^{<w}$ for any $w$. Indeed, $\RHom_\X(i_\ast F,i_\ast G) \simeq \RHom(i^\ast i_\ast F,G)$, so it suffices to show that for $F \in D^- \Coh(\S)^{\geq w}$, $i^\ast i_\ast F \in D^-\Coh(\S)^{\geq w}$ as well. This follows from the fact that $i^\ast i_\ast F$ has a filtration whose associated-graded is isomorphic to $\op{Sym}(\bL_{\S/\X}) \otimes F$, which can be proved following the proof of \cite[Lemma 3.21]{halpern2012derived} verbatim.

Let $\cC$ denote the subcategory of $F \in D^- \Coh_\S (\X)$ which admit a factorization $F' \to F \to F''$ with $F' \in D^- \Coh_\S (\X)^{\geq w}$ and $F'' \in D^- \Coh_\S (\X)^{<w}$. Then $\cC$ is triangulated and contains the essential image of $D^- \Coh(\S)$ under $i_\ast$, and therefore contains $D^b \Coh_\S (\X)$. The factorizations define a baric decomposition of $\cC$, and are functorial. Furthermore, the truncation functors preserve connective objects. Writing any $F$ as a filtered colimit $F = \colim F^\alpha$ such $F^\alpha \in D^b\Coh_\S(\X)$ and $\tau_{\leq n} F^\alpha$ stabilizes for any fixed $n$, we see that $\colim \radj{w} F^\alpha$ converges to an element of $D^- \Coh_\S(\X)^{\geq w}$ and likewise for $\colim \ladj{w} F^\alpha$.
\end{proof}

We denote the baric truncation functors on $D^-\Coh_\S(\X)$ by $\radj{w}$ and $\ladj{w}$ as well.
\begin{rem} \label{rem:truncating_category}
Let $F \in \QC_\S(\X)$. The argument of \cite[Lemma 3.30]{halpern2012derived} applies verbatim to show that if $H_n(F) \in \QC_\S(\X)^{<w}$ for all $n$, then $F \in \QC_\S(\X)^{<w}$, and the converse holds if $F \in \QC(\X)^{<w}_{<\infty}$.
\end{rem}

\begin{lem}[Quantization commutes with reduction]
If $F \in D^-\Coh(\X)^{\geq w}$ and $G \in \QC(\X)_{<\infty}$ with $i^\Qshriek G \in \QC(\S)^{<w}$, then the restriction map
$$\RHom_\X(F,G) \to \RHom_{\X^{ss}}(F|_{\X^{ss}},G|_{\X^{ss}})$$
is an equivalence.
\end{lem}

\begin{proof}
The proof of \cite[Theorem 3.29]{halpern2012derived} applies verbatim.
\end{proof}

\begin{proof}[Proof of \autoref{thm:derived_Kirwan_surjectivity}]

To extend \autoref{lem:baric_decomposition_supports} to a semiorthogonal decomposition of $D^- \Coh(\X)$, we must show that the inclusion $D^-\Coh_\S(\X)^{\geq w} \subset D^-\Coh_\S(\X)$ admits a right adjoint and $D^-\Coh_\S(\X)^{< w} \subset D^-\Coh_\S(\X)$ admits a left adjoint. Note that $\X - \S \subset \X$ is the preimage of the complement of a KN stratum in the classical quotient stack $\inner{\Spec}_X (\op{Sym} U_0) / G$, and this combined with \cite[Lemma 3.22]{halpern2012derived} implies the existence of a Koszul system $K_0 \to K_1 \to \cdots$ of perfect complexes supported on $\S$ with $\colim K_n \simeq R\inner{\Gamma}_\S \cO_{\X}$. They are perfect of uniformly bounded Tor-amplitude, and for any fixed $w$, $\op{Cone}(K_n \to \cO_X)|_\Z \in \Perf(\Z)^{<w}$ for all $n\gg 0$.  We define functors for $F \in D^- \Coh(\X)$
\begin{gather*}
\rgs{\geq w} (F) := \colim \radj{w} (K_n \otimes F), \text{ and } \\
\rgs{<w} (F) := \lim \ladj{w} (K_n^\dual \otimes F)
\end{gather*}

We must show that these functors land in $D^-\Coh_\S(\X)^{\geq w}$ and $D^-\Coh_\S(\X)^{<w}$ respectively. Note that by construction, the baric truncation functors $\radj{w}$ and $\ladj{w}$ preserve connective objects in $D^- \Coh_\S(\X)$. Therefore, if $P \to F$ is a morphism from a perfect object which is a quasi-isomorphism in low homological degree, then $R\inner{\Gamma}_\S^{\geq w} (P) \to R\inner{\Gamma}_\S^{\geq w} (F)$ is an isomorphism in low homological degree as well (and likewise for $R\inner{\Gamma}_\S^{< w}$). It thus suffices to assume $F$ is perfect. In this case, $F \in D^- \Coh (\X)^{\geq w}$ and $D^- \Coh(\X)^{<v}$ for some $w$ and $v$, and the colimit and limit defining $R\inner{\Gamma}_\S^{\geq w}$ and $R\inner{\Gamma}_\S^{< w}$ stabilize (as in \cite[Lemma 3.37,3.38]{halpern2012derived}). It follows that $R\inner{\Gamma}_\S^{\geq w} (F) \in D^-\Coh_\S(\X)^{\geq w}$ and $R\inner{\Gamma}_\S^{< w} (F) \in D^-\Coh_\S(\X)^{< w}$.

We have a canonical morphism $\rgs{\geq w} F \to F$, so in order to show that $\rgs{\geq w}$ is right adjoint to the inclusion $D^-\Coh_\S(\X)^{\geq w} \subset D^- \Coh(\X)$, it suffices to show that the cone of this morphism lies in the right orthogonal complement, which one can verify is $D^-\Coh(\X)^{<w}$. By Remark \ref{rem:truncating_category} and the right-exactness of $\radj{w}$, it suffices to prove that for $P \in \Perf(\X)$, $\op{Cone}(\rgs{\geq w} P \to P) \in D^-\Coh_\S(\X)^{<w}$. For $n$ sufficiently large, we have $\rgs{\geq w} (P) \simeq \radj{w}(K_n \otimes P)$, and the canonical morphism is the composition $\radj{w}(K_n \otimes P) \to K_n \otimes P \to P$. The cone of the first morphism lies in $D^- \Coh_\S(\X)^{<w}$ by construction, and the cone of the second morphism lies in $D^-\Coh_\S(\X)^{<w}$ by the properties of the Koszul system (having chosen $n$ large enough). Thus we conclude by the octahedral axiom.

The proof that $\rgs{<w}$ is left adjoint to the inclusion $D^- \Coh_\S(\X)^{<w} \subset D^-\Coh_\S(\X)$ follows the same pattern: One must show that $\op{Cone}(F \to \rgs{<w} F) \in D^-\Coh(\X)^{\geq w}$. First one observes that by approximating in low homological degree, it suffices to prove this for $P \in \Perf(\X)$. Then the canonical morphism can be factored as $F \to K_n^\dual \otimes F \to \ladj{w} (K_n^\dual \otimes F) \simeq \rgs{<w} F$ for $n \gg 0$, and the cone of each morphism in the composition lies in $D^-\Coh_\S(\X)^{\geq w}$.

Now that we have shown right (respectively left) admissibility of $D^-\Coh_\S(\X)^{\geq w}$ (respectively $D^-\Coh_\S(\X)^{<w}$), it follows that we have the semiorthogonal decomposition of \autoref{eqn:main_SOD}. Although, one must use the adjunctions $i^\ast \dashv i_\ast$ and $i_\ast \dashv i^\Qshriek$ to identify $\cG_w$ as defined above with the middle term of this semiorthogonal decomposition.

Finally we consider the baric decomposition of $D^-\Coh_\S(\X)$ established in \autoref{lem:baric_decomposition_supports}. First note that the argument for why $i^\ast i_\ast (D^-\Coh(\S)^{\geq w}) \subset D^-\Coh(\S)^{\geq w}$ actually shows that we have a natural isomorphism $\ladj{w+1} i^\ast i_\ast F \simeq F$ as functors from the subcategory $D^- \Coh(\S)^w$. The fully-faithfulness can be deduced from the semiorthogonal decomposition of $D^-\Coh(\S)$: for $F,G \in D^-\Coh(\S)^w$
\begin{align*}
\RHom_\X(i_\ast F,i_\ast G) &\simeq \RHom_{\QC(\S)}(i^\ast i_\ast F, G) \\
&\simeq \RHom_{\QC(\S)}(\ladj{w+1} i^\ast i_\ast F, G) \\
&\simeq \RHom_{\QC(\S)}(F,G)
\end{align*}
The proof that $\pi^\ast : D^-\Coh(\Z)^w \to D^-\Coh(\S)^w$ is an equivalence is similar, the inverse to $\pi^\ast$ being given by $\radj{w} \pi_\ast$.

To show that the baric decomposition is complete, note that there is some $k$ such that $K_n$ is $k$-connective for all $n$. Because  $\radj{w}$ and $\ladj{w}$ preserve connective objects in $D^-\Coh_\S(\X)$, it follows that for any morphism $P \to F$ such that $\tau_{\leq p} P \to \tau_{\leq p} F$ is a quasi-isomorphism, we have that $\tau_{\leq p-k} \rgs{\geq w} P \to \tau_{\leq p-k} \rgs{\geq w} F$ is a quasi-isomorphism as well.

For any $F \in D^-\Coh_\S(\X)$, we can choose a morphism from a perfect complex $P \to F$ which is a quasi-isomorphism after applying $\tau_{\leq p}$. However, for a perfect complex such that $\tau_{\leq p} P \neq 0$, we can always find a $w$ such that $\tau_{\leq p} \rgs{\geq w} P \neq 0$ and a $v$ such that $\tau_{\leq p} \rgs{< v} P \neq 0$. Thus for any non-zero object $F \in D^-\Coh_\S(\X)$, there is some $w$ and $v$ such that $\ladj{w} F \neq 0$ and $\radj{v} F \neq 0$. Hence the baric decomposition is complete.
\end{proof}

One can prove a version of \autoref{thm:derived_Kirwan_surjectivity} for the category $\Perf(\X)$ provided that $\cO_\S$ is perfect as a $\cO_\X$-module. This is a categorified version of Kirwan surjectivity.
\begin{ex}
If $f^\ast ( (\bL_\X|_\Z)^{>0}) \in (k'\Mod)_\heart$ for all points $f:\Spec (k') \to \Z$, where $k'$ is a finite extension of $k$, then $\cO_\S$ is a perfect $\cO_\X$-module.
\end{ex}
\begin{proof}
$\bL_{\S/\X}$ is an almost perfect complex, and so the condition implies that $(\bL_\X|_{\Z})^{>0}$ is perfect. The condition applies to the dual as well, so in fact $(\bL_\X|_\Z)^{<0}$ is a locally free sheaf. \autoref{lem:relative_cotangent_complex} implies that $\bL_{\S/\X}|_{\Z} \simeq (\bL_\X|_\Z)^{>0}[1]$, so $\bL_{\S/\X}$ is a shifted locally free sheaf as well, and $i : \S \to \X$ is a regular embedding.
\end{proof}

\begin{prop} \label{prop:DKS_perfect}
If $\cO_\S$ is a perfect $\cO_\X$-module, then we have a semiorthogonal decomposition
\[ \sod{\ldots, \Perf(\Z)^{w-1},\cG^{perf}_w,\Perf(\Z)^{w},\ldots}, \]
where $\cG^{perf}_w = \cG_w \cap \Perf(\X)$ is equivalent to $\Perf(\X^{ss})$ under the restriction functor.
\end{prop}

\begin{proof}
The hypothesis implies that $i$ has finite Tor amplitude, so $i_\ast$ maps $\Perf(\S)$ to $\Perf(\X)$. In particular, $\Perf_\S(\X)$ is generated up to retracts by the essential image of $i_\ast$, and so $\Perf_\S(\X)$ obtains a bounded baric decomposition for the same reason as above. It follows that for any $F \in \Perf(\X)$, we have $\rgs{\geq w} F \in \Perf_\S(\X)^{\geq w}$ and $\rgs{<w} F \in \Perf_\S(\X)^{<w}$, and thus the projection of $F$ to $\cG_w$ lands in $\cG_w^{perf}$, which establishes the semiorthogonal decomposition. To conclude that $\cG^{perf}_w \to \Perf(\X^{ss})$ is an equivalence, it suffices to show that it is essentially surjective. The image of the restriction functor generates $\Perf(\X^{ss})$ up to retracts, but $\cG^{perf}_w$ is idempotent complete by its definition, and $\cG^{perf}_w \to \Perf(\X^{ss})$ is fully faithful, so it follows that the restriction functor is essentially surjective.
\end{proof}

\section{The case of quasi-smooth stacks}

In this section we establish a modified version of \autoref{thm:derived_Kirwan_surjectivity} when $\X$ is quasi-smooth. The key observation in this context is the following:

\begin{lem} \label{lem:finite_tor_amplitude}
When $\X$ is quasi-smooth, then $\S$ and $\Z$ are quasi-smooth, and the projection $\pi : \S \to \Z$ is quasi-smooth as well.
\end{lem}
\begin{proof}
\autoref{lem:relative_cotangent_complex} shows that $\bL_\S|_\Z \simeq (\bL_\X|_\Z)^{\leq 0}$ is perfect and concentrated in degrees $-1$,$0$, and $1$, hence $\bL_\S$ is perfect with fiber homology concentrated in those degrees as well. An argument completely analogous to the proof of \autoref{lem:relative_cotangent_complex} shows that we have an equivalence of exact triangles in $D^-\Coh(\Z)$
\[
\xymatrix{ (\sigma^\ast \bL_\S)^{< 0} \ar[r] \ar[d] & \sigma^\ast \bL_\S \ar[r] \ar@{=}[d] & (\sigma^\ast \bL_\S)^{\geq 0} \ar[r] \ar[d] & \\ \bL_{\Z/\S}[-1] \ar[r] & \sigma^\ast \bL_\S \ar[r] & \bL_\Z \ar[r] & }
\]
So $\bL_\Z$ is perfect with fiber homology concentrated in degree $-1,0,$ and $1$, and thus $\Z$ is quasi-smooth as well. Finally, considering the fiber sequence of cotangent complexes coming from the composition $\Z \xrightarrow{\sigma} \S \xrightarrow{\pi} \Z$, we have an isomorphism $\sigma^\ast \bL_{\pi : \S \to \Z} \simeq \bL_{\Z/\S}[-1]$, so this is a perfect complex with fiber homology in degree $-1,0,$ and $1$ as well.

\end{proof}

We shall prove
\begin{thm} \label{thm:derived_Kirwan_surjectivity_quasi-smooth}
Let $\X$ be a quasi-smooth quotient stack containing a $\Theta$-stratum $\S \subset \X$. Assume that $H_1(f^\ast((\bL_\X|_\Z)^{<0})) = 0$ for every $k'$-point $f : \Spec k' \to \Z$, where $k'/k$ is a finite field extension. Then for any $w \in \bZ$, there is a semiorthogonal decomposition\footnote{This is an infinite semiorthgonal decompositions in the usual sense, so that every $F \in D^b\Coh(\X)$ lies in a subcategory generated by finitely many semiorthogonal factors.}

\begin{equation} \label{eqn:main_SOD_quasi_smooth}
D^b \Coh(\X) = \sod{ \underbrace{\ldots, D^b\Coh(\Z)_{w-1}}_{D^b\Coh_\S(\X)^{<w}},\cG^b_w, \underbrace{D^b \Coh(\Z)_{w},D^b \Coh(\Z)_{w+1},\ldots}_{D^b\Coh_\S(\X)^{\geq w}} },
\end{equation}
where we have identified $D^b \Coh_\S(\X)_w$ with the essential image of the fully faithful functor $i_\ast \pi^\ast : D^b\Coh(\Z)_w \to D^b\Coh(\X)$, and $\cG^b_w := \cG_w \cap D^b\Coh(\X)$ is equivalent to $D^b \Coh(\X^{ss})$ under the restriction functor.
\end{thm}

\begin{ex}
Consider the derived cotangent stack of a smooth global quotient, $\X = T^\ast (X/G)$. This has an explicit description as the quotient of the derived scheme $R\Spec_X(\op{Sym}(\fg \to \Omega^1_X))$ by the action of $G$, where $\Omega^1_\X$ is in degree $0$, and $\fg \to \Omega^1_X$ is the co-moment map. For any point $x \in \X$, $H_{1}(\bL_\X|_x)$ is the lie algebra of the stabilizer subgroup. If $\X$ has a KN-stratification, then this will automatically lie in the parabolic subgroup $P_\lambda$, and hence the hypotheses of \autoref{thm:derived_Kirwan_surjectivity_quasi-smooth} are satisfied.
\end{ex}

\begin{ex} We observe that the semiorthogonal decomposition of \autoref{thm:derived_Kirwan_surjectivity} can fail to preserve $D^b\Coh(\X)$ for quasi smooth $\X$, so some hypothesis in \autoref{thm:derived_Kirwan_surjectivity_quasi-smooth} is necessary. Consider a linear $\Gm$ action on $\bA^n$ having positive and negative weights, and choose $\lambda(t) = t$ so that the unstable subspace is the subspace with positive weights. Let $X \subset \bA^n$ be the hypersurface defined by an homogeneous polynomial $f$ of negative weight, and assume that $f(0)=0$. Finally assume that $Crit(f) \cap \bA^n_{+}$ is not isolated from the rest of $Crit(f)$.

Under these assumptions there is an object $F \in D^b\Coh(X^{ss}/\Gm)$ which fails to be perfect in any neighborhood of $S \subset X$, and thus for any extension to $D^b\Coh(X/\Gm)$, the derived restriction to $\{0\}$ must be unbounded. Let $\tilde{F}$ be the unique object of $\cG_w \subset D^-\Coh(\X)$ such that $\tilde{F}|_{X^{ss}/\Gm} \simeq F$. If $\tilde{F} \in D^b\Coh(\X)$, then it would admit a presentation as a right-bounded complex of locally free sheaves which was eventually periodic up to a shift by a character of $\Gm$ -- specifically $F_n = F_{n+2}(\op{wt}(f))$. This contradicts the fact that $\tilde{F}|_{\{0\}}$ has weight $\geq w$ and is unbounded.
\end{ex}

One can use \autoref{thm:derived_Kirwan_surjectivity_quasi-smooth} to study the change in $D^b\Coh$ under a variation of GIT quotient. Specifically, we consider a balanced wall crossing in the language of \cite{halpern2012derived} (An elementary wall crossing in the language of \cite{ballard2012variation}).

\begin{cor} \label{cor:wall_crossing}
Let $\X$ be a quasi-smooth quotient stack containing two $\Theta$-strata $\S_\pm$ such that $\Z_+ \simeq \Z_-$ and $\lambda_+(t) = \lambda_-(t^{-1})$ in a quotient presentation of $\X$. Assume that $H_1(f^\ast \bL_\X)$ has weight $0$ for every $f : \Spec k' \to \Z$ for finite extensions $k'/k$. Let $c$ be the $\lambda_+$-weight of $\det (\bL_\X|_\Z)$.
\begin{enumerate}
\item If $c = 0$, then $D^b\Coh(\X^{ss}_+) \simeq D^b\Coh(\X^{ss}_-)$.
\item If $c>0$, then there is a fully faithful embedding $D^b\Coh(\X^{ss}_+) \subset D^b\Coh(\X^{ss}_-)$.
\item If $c<0$, then there is a fully faithful embedding $D^b\Coh(\X^{ss}_-) \subset D^b\Coh(\X^{ss}_+)$.
\end{enumerate}
\end{cor}

\begin{rem}
As in the smooth case, the semiorthogonal complement to the embeddings in (2) and (3) has a further semiorthogonal decomposition into a number of categories of the form $D^b\Coh(\Z)_w$.
\end{rem}

\begin{proof}
The argument is the same as in \cite{halpern2012derived}. Applying \autoref{thm:derived_Kirwan_surjectivity_quasi-smooth} on either side of the wall, the only difference in the description of $\cG^b_w$ is the size of the window, $a$ (one also has to modify $w$ to account for the fact that $\lambda_+ = \lambda_-^{-1}$). In light of \autoref{lem:relative_dualizing_complex}, the difference between the value of $a$ on either side of the wall is the weight of $\det(\bL_\X|_\Z)$.
\end{proof}

\subsection{Remarks on Serre duality for quasi-smooth stacks}

We recall the key constructions of Serre duality in the derived setting from \cite[Section 4.4]{drinfeld2013some}. For any quasi-compact derived algebraic stack with affine stabilizers, there is a complex $\omega_\X \in D^b\Coh(\X)$ such that $\bD_\X(\bullet) := \inner{\RHom}_{QC}(\bullet,\omega_\X)$ induces an equivalence $D^b\Coh(\X) \to D^b\Coh(\X)^{op}$. In fact we have $\omega_\X = \pi^! (k)$ where $\pi$ is the projection to $\Spec k$ and $\pi^! : \icoh(\Spec k) \to \icoh(\X)$ is the shriek pullback functor on ind-coherent sheaves.

\begin{lem} \label{lem:absolute_dualizing complex}
Let $\X$ be a quasi-projective quotient stack which is quasi-smooth, then $$\omega_{\X} \simeq \det(\bL_\X)[\op{rank} \bL_\X].$$
\end{lem}

\begin{proof}
$\X$ admits a closed immersion into a smooth stack $i : \X \hookrightarrow \Y$. Any such closed immersion is a derived regular embedding, so we have $i^!(\bullet) \simeq i^\ast(\bullet) \otimes \det (\bL_{\X / \Y})[\rank \bL_{\X/\Y}]$ by translating \cite[Part IV.4, Corollary 14.3.2]{gaitsgory2015study} into our notation. Furthermore, in our notation \cite[Part IV.4, Proposition 14.3.4]{gaitsgory2015study} says that $f^!(\bullet) \simeq f^\ast(\bullet) \otimes \det(\bL_\Y)[\rank \bL_\Y]$ for the smooth morphism $f: \Y \to \Spec k$. Using the fiber sequence for the relative cotangent complex of a composition of morphisms, it follows that the formula $f^!(\bullet) \simeq f^\ast(\bullet) \otimes \det(\bL_f) [\rank \bL_f]$ is closed under composition of morphisms, hence it holds for $\X \to \Spec k$.
\end{proof}

The inclusion $\Perf(\X) \subset D^b\Coh(\X)$ defines a functor $\Xi_\X : \QC(\X) \to \icoh(\X)$ which is fully faithful and left adjoint to the tautological functor $\Psi_\X : \icoh(\X) \to \QC(\X)$ coming from the inclusion $D^b\Coh(\X) \subset \QC(\X)$ (see \cite[Section 1.5]{gaitsgory2013ind}). Regarding $\icoh(\X)$ as a module category for $\QC(\X)$, we have $\Xi_\X(F) \simeq F \otimes \Xi_\X(\cO_\X)$, where $\Xi_\X(\cO_\X)$ is just $\cO_\X$ regarded as an object of $D^b\Coh(\X)$. Note that $\Xi_\X$ provides a second way of identifying $D^b\Coh(\X)$ as a full subcategory of $\icoh(\X)$, which differs from the defining embedding as the compact objects. The fully faithful embedding $\Xi_\X : \Perf(\X) \to \icoh(\X)$ agrees with the inclusion $\Perf(\X) \subset D^b\Coh(\X) \subset \icoh(\X)$, however, so there is no danger of confusion when regarding a perfect complex as an object of $\QC(\X)$ or of $\icoh(\X)$.

The fact that $\bD_\X$ preserves $\Perf(\X) \subset D^b \Coh(\X)$ allows us to restrict it to a duality functor $\Perf(\X) \to \Perf(\X)^{op}$ which commutes with $\Xi_\X$ by construction and differs from the naive (i.e. linear) duality by tensoring with an invertible complex. This extends to an isomorphism $\QC(\X) \to \QC(\X)^\dual$, where the latter denotes the dual presentable stable $\infty$-category. This duality interchanges $D^-\Coh(\X)$ and $D^+\Coh(\X)$.

In our setting, $\S$ and $\X$ are both quasi-smooth, but the morphism $i : \S \hookrightarrow \X$ need not have finite Tor-dimension; this happens if and only if $\cO_\S$ is perfect as an $\cO_\X$-module. The morphism $j : \S \to \X$ need not be Gorenstein in the sense of derived algebraic geometry \cite{gaitsgory2013ind}, and the pullback functor $j^{\icoh,\ast} : \icoh(\X) \to \icoh(\S)$ need not be defined. Nevertheless we have

\begin{lem} \label{lem:relative_dualizing_complex}
Let $i : \S \hookrightarrow \X$ be a closed immersion of quasi-smooth quotient stacks, then there is an isomorphism $i^! \circ \Xi_\X(F) \simeq \Xi_\S ( i^\ast (F) \otimes \det \bL_{\S/\X})[\op{rank} \bL_{\S/\X}])$ which is functorial for $F \in \QC(\X)$.
\end{lem}

\begin{proof}
The fact that $i^!$ is a functor of $\QC(\X)$-module categories, as are $\Xi_\S$ and $\Xi_\X$, we need only prove the claim for $F = \cO_\X$. We make use of the formula $\bD_{\X} \circ i_\ast \simeq i_\ast \circ \bD_{\S}$, as functors from $D^b\Coh(\S) \to D^b\Coh(\X)$ (the proof of \cite[Corollary 9.5.9]{gaitsgory2013ind} in the case of schemes works verbatim). For $F \in D^b\Coh(\S)$, we have
\begin{align*}
\RHom_{\icoh(\X)}(i_\ast F , \cO_\X) &\simeq \RHom_{\icoh(\X)}(\omega_\X,i_\ast \bD_\S(F)) \\
&\simeq \RHom_{\QC(\X)}(\omega_\X,\Psi_\X(i_\ast \bD_\S(F)))
\end{align*}
In a slight abuse of notation we have regarded $\omega_\X \in D^b\Coh(\X)$ as an object of $\QC(\X)$ as well, implicitly using the fact that the functor $\Xi_\X$ is fully faithful. $\Psi_\X$ commutes with pushforwards of bounded coherent complexes, so we can identify the latter with $\RHom_{\icoh(\X)}(\Xi_\X(i^\ast \omega_\X),\bD_\S(F))$, which we can in turn identify with $\RHom_{\icoh(\X)}(F,\Xi_\S(\bD_\S(i^\ast \omega_\X)))$. This equivalence is functorial in $F \in D^b\Coh(\X)$, and thus because $i^!$ is right adjoint to $i_\ast$ we have $i^! \cO_\X \simeq \Xi_\S(\bD_\S(i^\ast \omega_\X)) \simeq \Xi_\S(\omega_\S \otimes i^\ast \omega_X^\dual)$. The result now follows from \autoref{lem:absolute_dualizing complex}.

\end{proof}

\subsection{Proof of \autoref{thm:derived_Kirwan_surjectivity_quasi-smooth}}

Throughout this subsection, we will assume the hypothesis of \autoref{thm:derived_Kirwan_surjectivity_quasi-smooth} holds. This implies that $\bL_{\Z/\S} \simeq (\sigma^\ast \bL_\S)^{<0}[1] \simeq (\sigma^\ast \bL_\X)^{<0}[1]$ is perfect with fiber homology in degrees $1,0,-1$. In particular it implies that $\sigma : \Z \to \S$ has finite Tor amplitude.

\begin{lem} \label{lem:quasi_smooth_baric_decomp}
The baric truncation functors of \autoref{lem:baric_decomp} and \autoref{lem:baric_decomposition_supports} induce bounded baric decompositions of $D^b \Coh(\S)$ and $D^b\Coh_\S(\X)$.
\end{lem}

\begin{proof}
Because $\sigma$ has finite Tor-amplitude, $\sigma^\ast F \in D^b\Coh(\Z)$ for any $F \in D^b\Coh(\S)$. Thus $\sigma^\ast F$ decomposes as a direct sum of objects in $D^b\Coh(\Z)^w$ for finitely many $w$. By Nakayama's lemma and the fact that $\sigma^\ast$ is compatible with the baric truncation functors on $D^-\Coh(\S)$ and $D^-\Coh(\Z)$, it follows that $F$ decomposes under the baric decomposition of $D^- \Coh(\S)$ as an iterated extension of objects in $D^-\Coh(\S)^w$ for finitely many $w$. Furthermore these objects lie in $D^b\Coh(\S)^w$, because $\sigma^\ast : D^- \Coh(\S)^w \to D^-\Coh(\Z)^w$ is an equivalence with inverse given by $\pi^\ast$. It follows that the baric truncation functors preserve $D^b\Coh(\S)$.

The boundedness of the baric decomposition of $D^b\Coh(\S)$ follows again from Nakyama's lemma, the compatibility of $\sigma^\ast$ with the baric decomposition on $D^b\Coh(\Z)$, and the boundedness of the latter. Once one has the baric decomposition of $D^b\Coh(\S)$, the fact that $\radj{w}$ and $\ladj{w}$ preserve $D^b\Coh_\S(\X)$ follows from the fact that $D^b\Coh_\S(\X)$ is generated by $i_\ast D^b\Coh(\S)$ under shifts and cones, and the boundedness of the baric decomposition of $D^b\Coh_\S(\X)$ follows likewise.
\end{proof}

\begin{lem} \label{lem:duality}
Let $a$ be the $\lambda$-weight of $\omega_\S|_\Z$. Then we have
\begin{enumerate}
\item $\bD_\X(D^b\Coh(\X)^{\geq w}) = D^b\Coh(\X)^{<a+1-w}$
\item $\bD_\X(D^b\Coh(\X)^{< w}) = D^b\Coh(\X)^{\geq a+1-w}$  
\end{enumerate}
\end{lem}

\begin{proof}
(1) and (2) are equivalent by application of $\bD_\X$. The fact that $\sigma : \Z \to \S$ has finite Tor amplitude implies that $\sigma^\ast \bD_\S(F) \simeq \inner{\RHom}_{\QC(\Z)}(\sigma^\ast F, \sigma^\ast \omega_\S)$. Indeed, $\bD_\S(F)$ can be approximated in high homological degree by $\bD_\S(P)$ for a perfect complex $P$, and in this case the identity follows from the dualizability of $P$. It follows from \autoref{lem:baric_decomp} that $\bD_\S$ flips $D^b\Coh(\S)^{\geq w}$ and $D^b\Coh(\S)^{<a+1-w}$.

$F\in D^b\Coh(\X)$ lies in $D^b\Coh(\X)^{<w}$ if and only if $\RHom_{\QC(\X)}(i_\ast P,F) = 0$ for all $P \in D^b\Coh(\S)^{\geq w}$, because $\Perf(\S)^{\geq w} \subset D^b\Coh(\S)^{\geq w}$. By Serre duality and the fact that $i_\ast \bD_\S \simeq \bD_\X i_\ast$, this is equivalent to $\RHom_{\QC(\S)}(i^\ast \bD_\X(F), \bD_\S(P)) = 0$ for all $P \in D^b\Coh(\S)^{\geq w}$. Using the fact that $D^- \Coh(\S)^{<a+1-w}$ is generated under limits by $D^b\Coh(\S)^{<a+1-w} = \bD_\S(D^b\Coh(\S)^{\geq w})$ (see \autoref{rem:homology}), this is equivalent to $i^\ast \bD_\X(F) \in D^- \Coh(\S)^{\geq a+1-w}$, or $\bD_\X(F) \in D^b\Coh(\X)^{\geq a+1-w}$.
\end{proof}

\begin{rem} \label{rem:alternate_characterization}
This lemma shows that in addition to the description
$$D^b\Coh(\X)^{\geq w} = \left\{ F \in D^b\Coh(\X) \left| \sigma^\ast i^\ast F \in D^-\Coh(\Z)_{w} \right.\right\},$$
following immediately from the definitions and \autoref{lem:baric_decomp}, we have an alternate characterization
$$D^b\Coh(\X)^{< w} = \left\{ F \in D^b\Coh(\X) \left| \sigma^\ast i^\ast \bD_\X F \in D^-\Coh(\Z)_{\geq a+1- w} \right.\right\}.$$
This description is convenient because it emphasizes the symmetry of the conditions defining $\cG^b_w$.
\end{rem}

\begin{lem} \label{lem:bounded_SOD}
$D^b \Coh (\X) = \bigcup_{v,w} \left( D^b \Coh(\X)^{<w} \cap D^b\Coh(\X)^{\geq v} \right)$.
\end{lem}

\begin{proof}
This is proved in essentially the same way as \cite[Lemma 3.36]{halpern2012derived} and the lemmas leading up to it.  In particular \autoref{lem:quasi_smooth_baric_decomp} implies that $K_n \otimes F \in D^b\Coh_\S (\X)^{\geq w}$ for some $w$, and so the weights of $\sigma^\ast i^\ast (K_n \otimes F) \simeq \sigma^\ast i^\ast(K_n) \otimes \sigma^\ast i^\ast F$ are bounded below. Because $K_n|_\Z \to \cO_\Z$ has negative weights, $\cO_Z$ is a summand of $K_n|_\Z$, and thus the weights of $\sigma^\ast i^\ast F$ are below as well. The claim now follows from \autoref{lem:duality}.
\end{proof}

\begin{proof} [Proof of \autoref{thm:derived_Kirwan_surjectivity_quasi-smooth}]
In order to get the semiorthogonal decomposition of \autoref{eqn:main_SOD_quasi_smooth} one must show that $\rgs{\geq w}$ and $\rgs{<w}$, the projection functors of \autoref{thm:derived_Kirwan_surjectivity_full}, preserve bounded objects. \autoref{lem:bounded_SOD} implies that for any $F \in D^b\Coh(\X)$ and $w \in \bZ$, the object $\op{Cone}(K_n \otimes F \to F) \in D^b\Coh(\X)^{<w}$ for $n \gg 0$, which implies that $\rgs{\geq w} F \simeq \radj{w} (K_n \otimes F)$. Likewise $\rgs{<w} F \simeq \ladj{w} (K_n^\dual \otimes F)$ for $n\gg 0$. It thus follows from \autoref{lem:quasi_smooth_baric_decomp}.

The semiorthogonal decompositions of $D^b\Coh_\S(\X)^{\geq w}$ and $D^b\Coh_\S(\X)^{<w}$ follow from the same in the general case of \autoref{thm:derived_Kirwan_surjectivity}. One only has to check that $i_\ast \pi^\ast : D^- \Coh(\Z)^{w} \to D^-\Coh(\X)$ preserves bounded coherent objects, which follows from \autoref{lem:finite_tor_amplitude}, and observe that $F \in D^-\Coh(\S)^w \simeq D^-\Coh(\Z)^w$ has bounded homology if and only if $i_\ast$ does. Also, one must check that every $F \in D^b\Coh_\S(\X)$ lies in the subcategory generated by finitely many of the semiorthogonal factors, which follows from \autoref{lem:bounded_SOD}.
\end{proof}

\section{Extensions to multiple strata, and local quotient stacks}
\label{sect:multiple_strata}

In this section we note that \autoref{thm:derived_Kirwan_surjectivity} and \autoref{thm:derived_Kirwan_surjectivity_quasi-smooth} extend to a more general setting. We will consider an open substack, $\X^{ss}$, of a local-quotient stack, $\X$, whose complement admits a $\Theta$-stratification. By this we mean a disjoint union of locally closed substacks covering $\X^{us} := \X \setminus \X^{ss}$, indexed by a set which is preordered by the real numbers (meaning we can assign a real number $\mu(\alpha)$ to every index $\alpha$). We assume that for every $c$ the subset $\bigcup_{\mu(\alpha)>c} |\S_\alpha| \subset |\X|$ is closed, that the closure of $|\S_\alpha|$ lies in $|\S_\alpha| \cup \bigcup_{\mu(\beta)>\mu(\alpha)} |\S_\beta|$.

In addition, we require that each $\S_\alpha$ is locally a derived $\Theta$-stratum in the following sense: We are given morphisms $\sigma_\alpha: \Z_\alpha \to \S_\alpha$ and $ i_\alpha : \S_\alpha \to \X$ along with open substacks $\U_p \subset \X$,  for each $p \in |\Z_\alpha|$, containing the image of $p$. We require that
\begin{itemize}
\item $\U_\alpha$ is a global quotient stack, and $i_\alpha : \S_\alpha \cap \U_p \to \U_p$ is identified with a derived $\Theta$-stratum in the sense of \autoref{defn:theta_stratum}.
\item For each point $p \in \Z_\alpha$ whose image lies in $\U_p \cap \U_q$, the map $(B\Gm)_{k'} \to \Z_\alpha$ coming from the identification of $\S_\alpha \cap \U_\alpha$ as a derived $\Theta$-stratum is isomorphic to the map coming from the identification of $\S_\alpha \cap \U_q \subset \U_q$ as a derived $\Theta$-stratum.
\end{itemize}
This additional structure is actually automatic from the perspective of identifying $\S_\alpha$ with an open substack of $\inner{\Map}(\Theta,\X)$, and that approach will be used in \cite{halpern2015derived} to avoid the need for such unwieldy local data.

\begin{rem}
In the case of a global quotient stack with multiple strata, one can fix a presentation for $\X$ as $\inner{\Spec}$ of a semifree equivariant sheaf of CDGA's, $\cO_X[U_\bullet;d]$, over a smooth quasi-projective $G$-scheme $X$ as in the previous sections. It is convenient to choose $X$ so that $\X \hookrightarrow \X'$ is a closed immersion, i.e. $\cO_X \to H_0(\cO_X[U_\bullet;d])$ is surjective. Then for each $\S_\alpha$, the open substack $\X' \setminus \bigcup_{\mu(\beta)>\mu(\alpha)} |\S_\beta|$ contains $|\S_\alpha|$, and we can restrict the algebra $\cO_X[U_\bullet;d]$ to this open substack to give a presentation of an open substack of $\X$ in which $\S_\alpha$ is a derived $\Theta$-stratum. These local quotient presentations tautologically meet the above compatibility criteria, so it is not necessary to choose different semi-free presentations of $\X$ in different local-quotient coordinate charts.
\end{rem}

\begin{defn} \label{defn:general_categories_local_quotient}
For any choice of integers $w_\alpha \in \bZ$ for each stratum, we define the categories
\begin{gather*}
D^-\Coh(\X)^{\geq w} := \left\{ F \in D^-\Coh(\X) \left| \forall \alpha, \forall p \in |\Z_\alpha|, F|_{\U_p} \in D^-\Coh(\U_p)^{\geq w_\alpha}  \right.\right\} \\
D^-\Coh(\X)^{< w} := \left\{ F \in D^-\Coh(\X) \left| \forall \alpha, \forall p \in |\Z_\alpha|, F|_{\U_p} \in D^-\Coh(\U_p)^{<w_\alpha} \right.\right\}
\end{gather*}
And in keeping with our previous conventions, we define
\begin{gather*}
D^-\Coh_{\X^{us}}(\X)^{\geq w} := D^-\Coh_{\X^{us}}(\X) \cap D^-\Coh(\X)^{\geq w} \\
D^-\Coh_{\X^{us}}(\X)^{< w} := D^-\Coh_{\X^{us}}(\X) \cap D^-\Coh(\X)^{< w} \\
\cG_w =  D^-\Coh(\X)^{< w} \cap D^-\Coh(\X)^{\geq w}
\end{gather*}
\end{defn}

\begin{thm} \label{thm:derived_Kirwan_surjectivity_full}
Let $\X$ be a local quotient stack, and let $\{\S_\alpha\}$ be a $\Theta$-stratification. Then for any choice $\{w_\alpha \in \bZ\}$ we have a semiorthogonal decomposition
\[
D^-\Coh(\X) = \sod{D^-\Coh_{\X^{us}}(\X)^{\geq w}, \cG_{w}, D^-\Coh_{\X^{us}}(\X)^{<w}},
\]
where $\cG_{w}$ is identified with $D^-\Coh(\X^{ss})$ via the restriction functor.
\end{thm}

\begin{rem}
As a direct consequence of \autoref{thm:derived_Kirwan_surjectivity}, the categories $D^-\Coh_{\X^{us}}(\X)^{\geq w}$ and $D^-\Coh_{\X^{us}}(\X)^{< w}$ admit further semiorthogonal decompositions analogous to baric decompositions. We have omitted this for the sake of brevity.
\end{rem}

\begin{rem}
In the case where $\X$ is quasi-smooth, this theorem holds with $D^b\Coh$ instead of $D^-\Coh$, following \autoref{thm:derived_Kirwan_surjectivity_quasi-smooth}.
\end{rem}

The key observation is the following
\begin{lem} \label{lem:SOD_for_limit_categories}
Let $\cC_i$ be a diagram of stable $\infty$ categories indexed by a category $I$. Assume that each $\cC_i$ has a semiorthogonal decomposition $\cC_i = \sod{\cA_i,\cB_i}$ such that for any morphism $f:i \to j$ the functor $f_\ast : \cC_i \to \cC_j$ maps $\cA_i$ to $\cA_j$ and maps $\cB_i$ to $\cB_j$. Then $\cC := \varprojlim_{i\in I} \cC_i$ admits a semiorthogonal decomposition $\cC = \sod{\cA,\cB}$, where $\cA = \varprojlim \cA_i$ and $\cB = \varprojlim \cB_i$.
\end{lem}

\begin{proof}
The fact that $f_\ast$ preserves the categories $\cA_i$ and $\cB_i$ implies that $f_\ast$ commutes with the projections onto $\cA_i$ and $\cB_i$. Indeed if $F \in \cC_i$ and $B \to F \to A \to$ is the exact triangle canonically identifying $A$ and $B$ as the projection of $F$ onto $\cA$ and $\cB$ respectively, then $f_\ast B \to f_\ast F \to f_\ast A \to$ is an exact triangle canonically identifying $f_\ast A$ and $f_\ast B$ with the projection of $f_\ast F \in \cC_j$ onto $\cA_j$ and $\cB_j$ respectively.

This implies that given an element $\{F_i \in \cC_i\} \in \cC$, the projections $\{A_i \in \cA_i\}$ define an object of $\varprojlim_i \cA_i$, and we have an exact triangle $\{B_i\} \to \{F_i\} \to \{A_i\} \to$. What remains is to show that $\varprojlim \cA_i$ is semiorthogonal to $\varprojlim \cB_i$ in $\cC$, which follows immediately from the description $\RHom_\cC (\{F_i\},\{G_i\}) \simeq \varprojlim \RHom_{\cC_i}(F_i,G_i)$.
\end{proof}

\begin{rem}
This holds for arbitrary finite semiorthogonal decompositions, and thus for baric decompositions.
\end{rem}

\begin{proof}[Proof of \autoref{thm:derived_Kirwan_surjectivity_full}]
Observe that for a derived $\Theta$-stratum in a global quotient stack, $\S \subset \U$, the baric decomposition on $\Perf(\S)$ is determined by the weights of restrictions along the maps $(B\Gm)_{k'} \to \S$ defined by points of $|\Z|$, and one can use this to recover the baric decomposition on $\QC(\S)$. From there, the categories $D^-\Coh_\S(\U)^{\geq w}$, $D^-\Coh_\S(\U)^{\geq w}$, and $\cG_w$ are uniquely determined by the data of the closed immersion $i : \S \to \U$. It follows that if $\U' \subset \U$ is an open substack such that $\S \cap \U' \to \U'$ is still a derived $\Theta$-stratum, then the restriction functor $D^-\Coh(\U) \to D^-\Coh(\U')$ preserves the subcategories in the semiorthogonal decomposition of \autoref{eqn:main_SOD}, provided that the maps $(B\Gm)_{k'} \to \S \cap \U'$ coming from the identification of the latter as a derived $\Theta$-stratum are isomorphic to maps $(B\Gm)_{k'} \to \S$ determined by the identification of $\S$ as a derived $\Theta$-stratum.

Let us consider the case of a single stratum $\S \subset \X$. We have a Zariski-cover of $\X$ consisting of $\X^{ss}$ as well as the substack $\U_p$ around each point in the image of $\Z \to \X$. Faithfully flat descent implies that if we let $\X' := \X^{ss} \sqcup \bigsqcup_{p \in |\Z|} \U_p$, and let $\X'_\bullet$ be the Cech nerve of the projection $\X' \to \X$, then $D^-\Coh(\X) \simeq \op{Tot} D^-\Coh(\X'_\bullet)$. In any multiple intersection of open substacks in this cover, $\S$ intersected with this open subset can still be identified as a (possibly empty) $\Theta$-stratum, so at every level of the Cech nerve $D^-\Coh(\X'_n)$ has the semiorthogonal decomposition of \autoref{eqn:main_SOD}. Furthermore, the restriction functors in this co-simplicial diagram of $\infty$-categories are compatible with these semiorthogonal decompositions by the compatibility hypotheses between the idnetifications of $\Theta$-strata in different $\U_p$. Thus by \autoref{lem:SOD_for_limit_categories} the semiorthogonal decomposition extends to $\QC(\X)$ as well.

Finally, extending the result to multiple strata is similar. Writing $\X$ as an ascending union of the open substacks $\X_{\mu \leq c} := \X^{ss} \cup \bigcup_{\mu(\alpha)\leq c} \S_\alpha$, we have $D^-\Coh(\X) = \varprojlim D^-\Coh(\X_{\leq c})$. The inductive argument for how to extend the semiorthogonal decomposition stratum by stratum is identical to that in \cite{halpern2012derived}, and the restrictions $D^-\Coh(\X_{\leq c})$ are compatible with these semiorthogonal decompositions, hence by \autoref{lem:SOD_for_limit_categories} we get a semiorthogonal decomposition of $D^-\Coh(\X)$ as well.
\end{proof}

\section{The virtual non-abelian localization theorem}

In this section we discuss a localization formula for computing the $K$-theoretic index $\chi(\X,F):= \sum (-1)^p \dim H_p(R\Gamma(\X,F))$ for $F \in D^b\Coh(\X)$. The quantity $\chi(\X,F)$ is only well defined if $\sum_p \dim H_p R\Gamma(\X,F) < \infty$. Our main result is technically independent of \autoref{thm:derived_Kirwan_surjectivity} and \autoref{thm:derived_Kirwan_surjectivity_quasi-smooth}, but it is closely related conceptually, and requires the same idea of equipping a classical KN-stratum with a derived structure such that \autoref{lem:relative_cotangent_complex} holds.

We shall consider a quasi-smooth stack $\X$, which for simplicity we assume is a derived global quotient stack, which has several derived $\Theta$-strata in the sense of \autoref{defn:theta_stratum}, indexed by some partially ordered set, so
\[
\X = \X^{ss} \cup \bigcup \S_\alpha, \text{ and } \overline{\S_\alpha} \subset \bigcup_{\beta \geq \alpha} \S_\beta
\]
For convenience, we shall introduce the notation $\bL_\alpha^+ := (\bL_\X|_{\Z_\alpha})^{>0} \in \Perf(\Z_\alpha)^{>0}$ and $\bL_\alpha^- := (\bL_\X|_{\Z_\alpha})^{<0} \in \Perf(\Z_\alpha)^{<0}$.

\begin{thm} \label{thm:nonabelian_localization}
Let $\X$ be a quasi-smooth stack admitting a derived $\Theta$-stratification as above. Let
$$E_\alpha = \op{Sym}(\bL^-_\alpha \oplus (\bL^+_\alpha)^\dual) \otimes (\det (\bL^+_\alpha))^\dual [-\op{rank} \bL_\alpha^+].$$
Then for any $F \in \Perf(\X)$,
\begin{equation} \label{eqn:virtual_nonabelian_localization}
\chi(\X,F) \simeq \chi(\X^{ss},F|_{\X^{ss}}) + \sum_\alpha \chi(\Z_\alpha, F |_{\Z_\alpha} \otimes E_\alpha),
\end{equation}
and the left side is well-defined whenever all of the terms on the right side have $\sum \dim H_p R\Gamma(\bullet) < \infty$.
\end{thm}

One of the key features of this formula is that the correction terms are a-priori coherent when $F \in \Perf(\X)$. Even though $E_\alpha$ does not have coherent homology sheaves, the component in each weight, $(E_\alpha)^w \in \Perf(\Z)^w$, is coherent and vanishes for all for $w \gg 0$. Thus only finitely many summands $(E_\alpha)^w$, those for which $(F|_{\Z_\alpha})^{-w}$ is nonvanishing, contribute to $\chi(\Z_\alpha,F|_{\Z_\alpha} \otimes E_\alpha)$.

\begin{lem} \label{lem:deformation_to_normal_cone}
Let $i:\S \subset \X$ be a closed immersion with perfect relative cotangent complex. Then for any $F \in \QC(\X)$, $R\inner{\Gamma}_\S F$ has a bounded below, increasing filtration whose associated graded is quivalent to
$$i_\ast \left( \op{Sym} (\bL_{\S/\X}^\dual [1]) \otimes i^! \cO_\X \otimes i^\ast F  \right).$$
\end{lem}

\begin{proof}
The functorial isomorphism $\rgs{} F \simeq F \otimes \rgs{} \cO_\X$, combined with the projection formula, shows that it is sufficient to prove the claim for $F = \cO_\X$. This is a consequence of the theory of derived deformation to the normal cone of \cite[Part IV.4, Section 10]{gaitsgory2015study}. They construct a sequence of derived square-zero extensions $\S = \S^{(0)} \hookrightarrow \S^{(1)} \hookrightarrow \cdots \hookrightarrow \X$ having the properties that
\begin{itemize}
\item $\S^{(n)} \hookrightarrow \S^{(n+1)}$ is a square-zero extension by (the pushforward of) $\op{Sym}^{n+1} (\bL_{\S/\X}[-1])$, and
\item $\colim \S^{(n)} \simeq \widehat{\X}_{\S}$ in the category of formal moduli problems under $\S$ \cite[Part IV.2]{gaitsgory2015study}.
\end{itemize}

The claim of the lemma will follow from the formula
\begin{equation} \label{eqn:cohomology_with_supports}
\rgs{} \cO_\X \simeq \colim_n \inner{\RHom}_{\QC(\X)} (\cO_{\S^{(n)}},\cO_\X)
\end{equation}
Indeed, exhibiting an object as the colimit of a diagram of the form $F_0 \to F_1 \to \cdots$ is the $\infty$-categorical generalization of having a (bounded below, increasing) filtration whose associated graded is $\bigoplus_n  \cofib(F_n \to F_{n+1})$. In our case we have a fiber sequence $i_\ast \op{Sym}^{n+1} (\bL_{\S/\X}[-1]) \to \cO_{\S^{(n+1)}} \to \cO_{\S^{(n)}}$, so the associated graded of the filtration is 
\[
\bigoplus_n \inner{\RHom}_{\QC(\X)}(i_\ast \op{Sym}^n (\bL_{\S/\X}[-1]),\cO_\X) \simeq i_\ast \bigoplus_n \inner{\RHom}_{\QC(\S)}(\op{Sym}^n(\bL_{\S/\X}[-1]),i^!\cO_\X).
\]
The claim follows from the dualizability of $\op{Sym}^n (\bL_{\S/\X}[-1])$.

In order to prove \autoref{eqn:cohomology_with_supports}, we use the identity $\colim \S^{(n)} \simeq \widehat{\X}_\S$, and hence $D^b\Coh(\X)_\S \simeq \colim_{n} D^b\Coh(\S^{(n)})$ by \cite[Part III.3, Corollary 1.1.6]{gaitsgory2015study}. It follows that we have a diagram
\[
\xymatrix{ \icoh(\S^{(n)}) \ar[drr]_{(i_n)_\ast} \ar[d] & & \\ \op{Ind}(\colim_n D^b\Coh(\S^{(n)})) \ar[rr]_{i_\ast} \ar[d]^{\simeq} & &  \icoh(\X) \\ \icoh(\X)_\S & &}
\]
Passing to right adjoints, \autoref{lem:categorical_nonsense} below implies that $i_\ast i^! \simeq \colim_n (i_n)_\ast i_n^!$, and we can consider the composition of this equivalence with $\Psi_\X : \icoh(\X) \to \QC(\X)$.

Now we have a fiber sequence of functors $i_\ast i^! \to \id_{\icoh(\X)} \to j_\ast j^{\op{IndCoh},\ast}$, where $j : U \to \X$ is the inclusion of the open complement of $\S$. Because the second two terms preserve the category $\icoh(\X)_{<\infty}$ and commute with $\Psi_\X$, it follows that $i_\ast i^!$ preserves eventually coconnective complexes, and under the equivalence $\Psi_\X : \icoh(\X)_{<\infty} \simeq \QC(\X)_{<\infty}$ we can identify $i_\ast i^! \cO_\X \simeq \rgs{} \cO_\X$. 

Furthermore, for each closed immersion $i_n : \S^{(n)} \to \X$, the functor $i_n^{\Qshriek}$ preserves eventually coconnective objects, so $i_n^{\Qshriek} : \QC(\X)_{<\infty} \to \QC(\S^{(n)})_{<\infty}$ is right adjoint to $(i_n)_\ast : \QC(\S^{(n)})_{<\infty} \to \QC(\X)_{<\infty}$. It follows that $i_n^!$ preserves eventually coconnective objects, and that $\Psi_\X(i_n^! F) \simeq i_n^\Qshriek F$ for $F \in \icoh(\X)_{<\infty}$. It follows that $\Psi_\X( (i_n)_\ast i_n^! \cO_\X) \simeq \inner{\RHom}_{\QC(\X)}(\cO_{\S^{(n)}},\cO_\X)$, and \autoref{eqn:cohomology_with_supports} follows.
\end{proof}

\begin{lem} \label{lem:categorical_nonsense}
Let $\Psi : I \to \op{Pr}^L$ be a diagram of presentable $\infty$-categories, and let $\cC = \colim_n \cC_n$ in $\op{Pr}^L$. If $i_n : \cC_n \to \cC$ is the functor defining $\cC$ as a colimit, then $\id_{\cC} \simeq \colim_{n \in I} i_n i_n^R$.
\end{lem}
\begin{proof}
The map $\colim_{n \in I} i_n i_n^R \to \id_\cC$ is induced by the counits of adjunction. For any objects $F,G \in \cC$, the induced map $\Map(F,G) \to \Map(\colim_{n\in I} i_n i_n^R (F),G)$ can be identified with the map $\Map(F,G) \to \varprojlim_{n\in I} \Map(i_n^R(F),i_n^R(G))$. Therefore, the claim follows from the fact that the map $i_n^R$ identify $\colim_I \Psi$ with the limit category $\varprojlim_I \Psi^R$ \cite[Proposition 1.7.5]{drinfeld2011compact}.
\end{proof}

\begin{lem} \label{lem:filter_pushforward}
Let $\S$ be a $\Theta$-stratum, with projection $\pi:\S \to \Z$ and section $\sigma : \Z \to \S$. Then for any $F \in \QC(\S)$, $\pi_\ast (F) \in \QC(\Z)$ has a bounded below, increasing filtration whose associated graded is $\op{Sym}(\bL_{\S/\Z}) \otimes \sigma^\ast F$.
\end{lem}

\begin{proof}
Let us recall the definition $\Z = R\Spec_Z \cB / L$. Then we can factor the morphism $\pi$ as an affine morphism $f : R\Spec_Z \cA / P \to R\Spec_Z \cB / P$,\footnote{Note that we can regard $\cA$ as a quasicoherent, and in fact locally semi-free, sheaf of CDGA's over $Z$ because the projection $Y \to Z$ is affine, and in fact an \'{e}tale-locally trivial bundle of affine spaces.} where $P$ acts through the projection $P \to L$, followed by the morphism $g : R\Spec_Z \cB / P \to \Z$.

Now any module has a presentation of the form $\cA \otimes_{\cO_Z} M$ with some differential determined by $d : M \to \cA \otimes M$. The pushforward $f_\ast$ is simply the functor which forgets the $\cA$-module structure (but remembers the action of $P$), and the pushforward $g_\ast$ have an explicit description as well. Because $R\Spec_Z \cB / P \to \Z$ is a $U$-gerbe, where $U$ is the unipotent group which is the kernel of $P \to L$, we can compute the pushforward using the lie-algebra cohomology. For any $E \in \QC(R\Spec_Z \cB / P)$ we have $g_\ast (E) \simeq E \otimes_k \bigwedge^\ast (\fu^\dual)$ regarded as an $L$-equivariant $\cB$-module, where the differential encodes the action of $U$.

Combining these two observations we have an explicit complex $\bigwedge^\ast (\fu^\dual) \otimes_k \cA \otimes M$ computing $\pi_\ast (F)$. Assume that we have replaced $\cA$ with a semi-free presentation over $\cO_\Z$. Then we can filter this complex by tensor order, considering both the semi-free generators of $\cA$ and the generators in $\fu^\dual$. One can check that the associated graded has an underlying $\cO_\Z$ module which is still isomorphic to $\bigwedge^\ast (\fu^\dual) \otimes_k \cA \otimes M$, but with a new differential such that the complex is equivalent to $\sigma^\ast(\op{Sym} (\bL_{\S/\Z}) \otimes F)$.
\end{proof}

\begin{proof} [Proof of \autoref{thm:nonabelian_localization}]
Given the previous lemmas, the proof is a straightforward application of cohomology with supports. We describe this explicitly for a single stratum, and the general case follows by a simple inductive argument. For any $F$ there is an exact triangle $R\inner{\Gamma}_\S(F) \to F \to j_\ast(F|_{\X^{ss}}) \to$, where $j : \X^{ss} \to \X$ is the inclusion. If $R\Gamma$ is finite dimensional for the outer two, then it is finite dimensional for $F$ as well, and applying $\chi$ we have $\chi(\X,F) = \chi(\X^{ss},F|_{\X^{ss}}) + \chi(\X,R\inner{\Gamma}_\S(F))$.

Now \autoref{lem:deformation_to_normal_cone} implies that $R\Gamma_\S(\X,F)$ has a bounded below, increasing filtration whose associated graded is $R\Gamma(\S,\bigoplus_n \op{Sym}^n (\bL_{\S/\X}^\dual [1]) \otimes i^!\cO_\X \otimes i^\ast F)$. Note that by \autoref{lem:relative_cotangent_complex} and \autoref{lem:relative_dualizing_complex}, we can rewrite this as
\[
\bigoplus_n R\Gamma \left( \S,\op{Sym}^n (\ladj{0} (i^\ast \bL_{\X}^\dual)) \otimes \det \left( \ladj{0} (i^\ast \bL_{\X}^\dual) \right)[\rank (\bL^+)^\dual] \otimes i^\ast F \right)
\]
Because the weights of $i^\ast F$ are bounded above, and the weights of $\op{Sym}^n (\ladj{0} (i^\ast \bL_{\X}^\dual))$ become increasingly negative as $n$ increases, only finitely many terms of this sum are nonzero.

We may identify the restriction of $\ladj{0}(i^\ast \bL_\X^\dual)$ to $\Z$ with $(\bL^+)^\dual$, and in the proof of \autoref{lem:finite_tor_amplitude} we computed that $\bL_{\pi : \S \to \Z} \simeq \bL_{\Z/\S}[-1] \simeq (\sigma^\ast \bL_\S)^{<0}$, so combining this with \autoref{lem:relative_cotangent_complex} we see that $\bL_{\S/\Z} \simeq \bL^-$. Applying \autoref{lem:filter_pushforward}, we see that each summand in the expression above has a bounded below increasing filtration whose associated graded is
\[
\bigoplus_m R\Gamma \left(\Z,\op{Sym}^n ( (\bL^+)^\dual) \otimes \op{Sym}^m (\bL^-) \otimes \det((\bL^+)^\dual)[\rank(\bL^+)^\dual] \otimes \sigma^\ast F \right)
\]
Again, because the weights of $\sigma^\ast F$ are bounded above, only finitely many terms of this sum are nonzero. It follows that only finitely many terms of the sum $R\Gamma(\Z,E_\alpha \otimes \sigma^\ast F)$ are non-zero, and if they are finite dimensional, then so is $R\Gamma_\S(\X,F)$, and $\chi(\X,R\Gamma_\S(\X,F)) = \chi(\Z,E_\alpha \otimes \sigma^\ast F)$.

\end{proof}

\bibliography{references_appendix}{}
\bibliographystyle{plain}

\end{document}